\documentclass[12pt,a4paper]{article}

\usepackage{mathrsfs}

\usepackage{amssymb, amsthm, amsfonts, amsxtra, amsmath}
\usepackage{latexsym}
\usepackage{mathabx,MnSymbol}

\theoremstyle{plain}

\newtheorem{Theorem}{Theorem}[section]
\newtheorem{Lem}[Theorem]{Lemma}
\newtheorem{Prop}[Theorem]{Proposition}
\newtheorem{Cor}[Theorem]{Corollary}
\newtheorem{Not}[Theorem]{Notation}

\newtheorem{Def-Prop}[Theorem]{Definition-Proposition}

\theoremstyle{definition}
\newtheorem{Def}[Theorem]{Definition}

\theoremstyle{remark}
\newtheorem*{Rem}{Remark}

\date{}

\DeclareMathOperator{\Index}{Index}
\DeclareMathOperator{\Ad}{Ad}
\DeclareMathOperator{\Irr}{Irr}
\newcommand{\omin}{\otimes_\mathrm{min}}
\newcommand{\algtensor}{\otimes_\mathrm{alg}}
\DeclareMathOperator{\smbox}{\Box}

\usepackage{parskip} 
\makeatletter 
\def\thm@space@setup{%
  \thm@preskip=\parskip \thm@postskip=0pt
}
\makeatother
\usepackage{enumitem}
\newlist{thmenum}{enumerate}{1}%
\setlist[thmenum]{label={\rm (\arabic*)}}%

\begin{document}

\hyphenation{Wo-ro-no-wicz com-pa-ti-bi-li-ty}

\title{A construction of finite index C$^*$-algebra inclusions from free actions of compact quantum groups}
\author{K. De Commer\\
\small Department of Mathematics, University of Cergy-Pontoise,\\
\small UMR CNRS 8088, F-95000 Cergy-Pontoise, France\\
\small e-mail: Kenny.De-Commer@u-cergy.fr \\
\and M. Yamashita\footnote{Supported in part by the ERC Advanced Grant 227458 OACFT ``Operator Algebras and Conformal Field Theory''}\\
\small Dipartimento di Matematica,\\
\small Universit\`{a} degli Studi di Roma ``Tor Vergata'',\\
\small Via della Ricerca Scientifica 1, 00133 Rome, Italy\\
\small e-mail: yamashit@mat.uniroma2.it}

\maketitle

\begin{abstract}
\noindent Given an action of a compact quantum group on a unital C$^*$-algebra, one can consider the associated Wassermann-type C$^*$-algebra inclusions. One hereby amplifies the
original action with the adjoint action associated with a finite dimensional unitary representation, and considers the induced inclusion of fixed point algebras. We show that this
inclusion is a finite index inclusion of C$^*$-algebras when the quantum group acts freely. Along the way, two natural definitions of freeness for a compact quantum group action, due
respectively to D. Ellwood and M. Rieffel, are shown to be equivalent.
\end{abstract}

\emph{Keywords}: compact quantum groups; C$^*$-algebras; Hilbert modules; free actions

AMS 2010 \emph{Mathematics Subject Classification}: 17B37, 81R50, 46L08


\section{Introduction}

One of the fundamental concepts in the study of locally compact quantum groups is the notion of `noncommutative principal bundles', or the free and proper actions on `noncommutative
spaces', the noncommutative spaces being represented by various algebraic structures. In the C$^*$-algebraic study of such principal bundles, it turned out that there can be two
different ways to formulate the freeness of an action.

The first is a certain density condition on the coaction map, called the \emph{Ellwood condition}, introduced by D.A.~Ellwood~\cite{Ell1}. If the algebra is commutative, so that we
are back in the classical case of a locally compact group acting continuously on a locally compact space, this condition is equivalent to freeness in the ordinary sense.  In the
purely algebraic setting, the Ellwood condition corresponds to the notion of a Hopf-Galois extension.

The second is the notion of a \emph{saturated} action, which is more suited for the study of K-theory of operator algebras. It was introduced by M.~Rieffel~\cite{Phi1,Rie1} in the
setting of actions by compact groups on C$^*$-algebras.  Since it is stated  as a condition on the structure of the crossed product algebra, there is a straightforward generalization
to the case of compact quantum group actions.  For example, the case of \emph{finite} quantum groups was studied by W.~Szyma\'nski and C.~Peligrad~\cite{Pel1}.

It has been known that these conditions are closely related to each other. For example, when $\mathbb{G}$ is a compact Lie group, C.~Wahl~\cite[Proposition 9.8]{Wah1} showed that they
are equivalent.

Our first main result is that the above two notions actually coincide in the setting of compact quantum group actions.

\begin{Theorem}\label{ThmFreeChar} Let $\mathbb{G}$ be a compact quantum group acting continuously on a C$^*$-algebra $A$. Then the following conditions are equivalent:
\begin{thmenum}\item The action satisfies the Ellwood condition.
\item The action is saturated.
\end{thmenum}
\end{Theorem}

Our following result relates the freeness of a compact quantum group action to certain ring-theoretical properties of the associated isotypical components.

Let $\pi$ be a finite dimensional unitary representation of a compact quantum group $\mathbb{G}$, $A$ a unital C$^*$-algebra acted upon freely by $\mathbb{G}$, and $A_{\pi}$ be the
isotypical component of $A$ associated to $\pi$. In particular, the isotypical component for the trivial representation is the fixed point subalgebra $A^{\mathbb{G}}$. Each $A_{\pi}$
becomes an $A^{\mathbb{G}}$-bimodule by the algebra structure of $A$.

The spaces $A_\pi$ can be interpreted as sections of a direct sum of the vector bundle induced by the representation $\pi$. In the classical case of compact group actions on compact
Hausdorff spaces, they are known to be finitely generated projective over the algebra of the base space. In our C$^*$-algebraic setting, we obtain the same result from a combination
of a technique used in the proof of Theorem~\ref{ThmFreeChar} and Kasparov's stabilization theorem for Hilbert C$^*$-modules.

\begin{Theorem}\label{TheoProj} Let $A$ be a unital C$^*$-algebra endowed with a free action of $\mathbb{G}$. Then each isotypical component $A_{\pi}$ is finitely generated projective
as a right $A^\mathbb{G}$-module.\end{Theorem}

We note that, for actions of discrete group duals on general C$^*$-algebras, this was proven by W. Szyma\'{n}ski in unpublished work. But even for compact groups acting on general
unital C$^*$-algebras, our result seems to be new.

Note that the conclusion in Theorem~\ref{TheoProj} also holds for \emph{ergodic} actions of $\mathbb{G}$, i.e.~ actions for which $A^{\mathbb{G}} = \mathbb{C}$. See~\cite{Hoe1} for
the case of compact groups, and~\cite{Boc1} for the general case of compact quantum groups. Our proof of Theorem~\ref{TheoProj} was inspired by the short argument for this result
which appears in~\cite{Pop1}, as well as by~\cite{Was1} and the first sections of~\cite{Bic1}. Our arguments will cover both the free and ergodic cases at once. We stress that, since
we work in the C$^*$-algebraic setting, we need to refine some of the von Neumann algebraic techniques which appear in the above papers.

The conclusion of Theorem~\ref{TheoProj} is not valid for an arbitrary action: one can check for example that the isotypical components for the action of the circle group on the
closed unit disc are not finitely generated over the fixed point algebra, except for the fixed point algebra itself. The problem is essentially that the field of stabilizer groups is
not continuous.

Subsequent to Jones initial subfactor paper~\cite{Jon1}, a lot of effort has gone into constructing von Neumann algebraic subfactors, starting from more classical symmetries and
building up further to quantum symmetries, see
e.g.~\cite{Was2},\cite{Wen1},\cite{Saw1},\cite{Ban1},\cite{Ban2},\cite{Nik1}.
Motivated by this celebrated theory, Y.~Watatani~\cite{Wat1} introduced the
notion of finite index inclusion of \emph{C$^*$-algebras}. We will show that,
also in the C$^*$-algebraic setting, the Wassermann type inclusion associated
with finite dimensional unitary representations of quantum groups provides an
example of such an inclusion.  This generalizes the case of finite groups
in~\cite[section 2.8]{Wat1} and finite quantum groups in~\cite{Pel1}.  The key
is that the above structure theorem on $A_\pi$ gives a finite quasi-basis for
this inclusion.

\begin{Theorem}\label{ThmFinIndIncl}
Let $\mathbb{G}$ act freely on a unital C$^*$-algebra $A$, and let $\pi$ be a finite dimensional unitary representation of $\mathbb{G}$. Consider $A\otimes B(\mathscr{H}_{\pi})$ with
its induced action by $\mathbb{G}$. Then the inclusion \[A^{\mathbb{G}} \subseteq (A\otimes B(\mathscr{H}_{\pi}))^{\mathbb{G}}\] is a finite index inclusion of C$^*$-algebras. When
$\pi$ is irreducible, the index of the natural conditional expectation is equal to the square of the quantum dimension of $\mathscr{H}_{\pi}$. \end{Theorem}

The paper is organized as follows. We gather basic facts about compact quantum groups and their representations in Section~\ref{SecAc}, where we also prove a crucial Pimsner-Popa type
estimate on the complete boundedness of the projections onto spectral subspaces. The most technical part of this paper occupies Section~\ref{adjt}, where we study the adjointability
of Galois maps in terms of various Hilbert C$^*$-bimodule structures of the isotypical components. The first two of our main theorems are proved in Section~\ref{free}, based on the
results of Section~\ref{adjt}.  Finally, we study the C$^*$-algebraic index for the Wassermann type inclusion associated with a free action and an irreducible representation in
Section~\ref{indx}.

\emph{General notations}

We denote the identity maps of various objects by $\iota$ once and for all. If $X$ is a Banach space and $E\subseteq X$ is a subset, we denote by $\lbrack E\rbrack$ the closed linear
span of $E$ inside $X$. Following the convention of right Hilbert C$^*$-modules, the scalar product of Hilbert spaces is taken to be conjugate linear in the first argument. The
complex conjugate of a Hilbert space $\mathscr{H}$ will be identified with the dual of $\mathscr{H}$ by means of the inner product and denoted by $\mathscr{H}^*$.  The multiplier
C$^*$-algebra of a C$^*$-algebra $C$ is denoted by $\mathcal{M}(C)$.  When $\mathcal{E}$ is a right Hilbert C$^*$-module over a C$^*$-algebra $A$, the algebra of (adjointable)
$A$-endomorphisms is denoted by $\mathcal{L}(\mathcal{E})_A$ and that of the compact $A$-endomorphisms is by $\mathcal{K}(\mathcal{E})_A$.  If there is no fear of confusion we also
write $\mathcal{L}(\mathcal{E})$ and $\mathcal{K}(\mathcal{E})$.

\section{Isotypical components of quantum group actions and associated Hilbert modules}\label{SecAc}

\subsection{Compact quantum groups}\label{SubSecCQG}

In this section, we review the theory of compact quantum groups.  A compact quantum group $\mathbb{G}$ is represented by a unital C$^*$-algebra $C(\mathbb{G})$, together with a unital
$^*$-homomorphism \[\Delta\colon C(\mathbb{G})\rightarrow C(\mathbb{G})\omin C(\mathbb{G})\] satisfying the coassociativity and the cancellation properties~\cite{Wor1,Mae1}. We will
denote by $P(\mathbb{G})\subseteq C(\mathbb{G})$ the Hopf algebra of matrix coefficients associated with $\mathbb{G}$, by $S$ its antipode, and by $\varphi\colon
C(\mathbb{G})\rightarrow \mathbb{C}$ the invariant Haar state on $C(\mathbb{G})$.

Let $\pi$ be a finite dimensional unitary representation of $\mathbb{G}$, by which we mean a finite dimensional Hilbert space $\mathscr{H}_{\pi}$ together with a left
$C(\mathbb{G})$-comodule structure \[\delta_{\pi}\colon \mathscr{H}_{\pi}\rightarrow C(\mathbb{G})\otimes \mathscr{H}_{\pi}\] satisfying $1\otimes v^*w = \delta_{\pi}(v)^*
\delta_{\pi}(w)$ for all $v,w\in \mathscr{H}_{\pi}$, having interpreted $\mathscr{H}_{\pi}$ as linear operators between the Hilbert spaces $\mathbb{C}$ and $\mathscr{H}_{\pi}$.
Choosing an orthogonal basis $\{e_i\}$ of $\mathscr{H}_{\pi}$, and writing $\delta_{\pi}(e_i) = \sum_j u_{ij}\otimes e_j$, this means that $\sum_{k} u_{ik}^* u_{jk} = \delta_{ij}$.
Consequently $\sum_k u_{ki}u_{kj}^* = \delta_{ij}$, as $u$ is invertible. We let $\Irr \mathbb{G}$ denote a complete representative system of irreducible finite dimensional unitary
representations of $\mathbb{G}$ up to unitary equivalence.

If $\pi$ is a unitary representation of $\mathbb{G}$, we denote by $\overline{\pi}$ the associated contragredient representation.  It is implemented on $\mathscr{H}_\pi^*$ with the
dual comodule structure, but we equip it with a new Hilbert space structure averaged out by means of $\varphi$. More precisely, choosing an orthonormal basis $e_i$ of
$\mathscr{H}_{\pi}$ with $\delta_{\pi}(e_i) =\sum_j u_{ij}\otimes e_j$, and writing $e_i^* = \langle e_i,\,\cdot\,\rangle$, we define the new scalar product on $\mathscr{H}_\pi^*$ by
\[
\left\llangle e_i^*,e_j^* \right\rrangle = \sum_{k,l} \varphi(u_{ik}u_{jl}^*) \langle e_k^*,e_l^*\rangle = \sum_k \varphi(u_{ik}u_{jk}^*).
\]

By Woronowicz's theory~\cite{Wor1}, we obtain the invertible positive matrix $Q_{\pi}\in B(\mathscr{H}_{\pi})$ for each $\pi \in \Irr \mathbb{G}$ satisfying $\textrm{Tr}(Q_{\pi}) =
\textrm{Tr}(Q_{\pi}^{-1})$ and
\begin{equation}\label{EqWorCharProperty}
\varphi(u_{ij}^*u_{kl}) = \delta_{ik} \frac{\langle e_l,Q_{\pi}e_j\rangle}{\textrm{Tr}(Q_{\pi})},\quad \varphi(u_{ij}u_{kl}^*) = \delta_{jl}\frac{\langle
e_k,Q_{\pi}^{-1}e_i\rangle}{\textrm{Tr}(Q_{\pi}^{-1})}, \qquad \textrm{for all }i,j,k,l.
\end{equation}
The number $\textrm{Tr}(Q_{\pi})$ is known as the \emph{quantum dimension} of $\mathscr{H}_{\pi}$,\[\dim_q(\mathscr{H}_{\pi}) = \textrm{Tr}(Q_{\pi}).\] One can then define the quantum
dimension of any representation of $\mathbb{G}$ by linearity. See e.g.~\cite{Rob1} for a detailed exposition.

\subsection{Actions of compact quantum groups}

An action of $\mathbb{G}$ on a (possibly non-unital) C$^*$-algebra $A$ is given by a non-degenerate injective $^*$-homomorphism $\alpha\colon A\rightarrow A\omin C(\mathbb{G})$,
satisfying the coaction property $( \alpha\otimes\iota) \circ \alpha = (\iota\otimes\Delta ) \circ \alpha$ and the density condition \[\lbrack \alpha(A)(1\otimes C(\mathbb{G}))\rbrack
= A\omin C(\mathbb{G}).\] We denote by $B = A^{\mathbb{G}}$ the C$^*$-algebra of $\mathbb{G}$-invariant elements, i.e.~ elements $x \in A$ satisfying $\alpha(x) = x  \otimes 1$. One
then has the following map $E_B$ from $A$ to $B$: \[E_B(x) = (\iota \otimes \varphi)\alpha(x), \qquad x\in A.\] In case $A$ is unital, this is a conditional expectation. In general,
$E_B$ is a c.c.p.~ $B$-bimodule map. 

When $\pi$ is a finite dimensional unitary representation of $\mathbb{G}$, we can  consider the vector space of `equivariant functions' \[ A \smbox \mathscr{H}_{\pi} = \{ z \in A
\otimes \mathscr{H}_{\pi} \mid (\alpha\otimes \iota) z = (\iota\otimes \delta_{\pi}) z \}.\] It has a natural $B$-bimodule structure, as well as a right $B$-valued Hermitian inner
product which is characterized by the following identity in $B\cong B\otimes \mathbb{C}$: \[ \langle w,z \rangle_B = w^*z,\qquad w,z \in A \smbox \mathscr{H}_{\pi}. \] We further put
\[A_{\pi}= \textrm{ linear span of }\{ (\iota \otimes \omega)z \mid z\in A \smbox \mathscr{H}_{\pi},\, \omega\in \mathscr{H}_{\pi}^*\} \subseteq A,\] which we call the
$\pi$-isotypical component of $A$. These $A_{\pi}$ are naturally $B$-bimodules with the right $B$-Hermitian inner product defined by \[ \langle x,y\rangle_B = E_B(x^*y).\] They carry
an (algebraic) right $P(\mathbb{G})$-comodule structure. Note that for the trivial representation $\pi = \mathrm{triv}$, we have $A_{\mathrm{triv}} = B$. The involution on $A$ and the
conjugate operation on the unitary representations of $\mathbb{G}$ are related by \[A_{\overline{\pi}} = \{x^* \mid x\in A_{\pi}\}.\]
We denote by $P(A)$ the $^*$-algebra $\sum_{\pi \in \Irr \mathbb{G}}^\oplus A_{\pi}$.

If $\pi_1$ and $\pi_2$ are two finite dimensional unitary representations of $\mathbb{G}$, we denote their tensor product representation on $\mathscr{H}_{\pi_1}\otimes
\mathscr{H}_{\pi_2}$ by $\pi_1\times \pi_2$. One then has the inclusion $A_{\pi_1}\cdot A_{\pi_2} \subseteq A_{\pi_1\times \pi_2}$.

We omit the proof of the following lemma, which follows from a straightforward calculation. In the statement of the lemma, we endow $\mathscr{H}^*$ with its modified Hilbert space
structure.

\begin{Lem}\label{LemMult} Let $\pi$ be an irreducible representation of $\mathbb{G}$. Then the map\[\phi_{\pi}\colon (A \smbox \mathscr{H}_{\pi})\otimes \mathscr{H}_{\pi}^*
\rightarrow A_{\pi}, z\otimes \omega \rightarrow \sqrt{n}(\iota\otimes \omega Q_{\pi}^{-1})z\] is an isomorphism of $B$-bimodules which is compatible with the $B$-valued inner
product.
\end{Lem}

\subsection{Hilbert module structures}

The above right $B$-modules $A_{\pi}$ and $A\smbox \mathscr{H}_{\pi}$ are complete with respect to their $B$-valued inner products. To see this, we first introduce the following
special elements.

\begin{Def} Let $\pi$ be an irreducible representation of $\mathbb{G}$. We call \emph{the quantum character} of $\pi$ the unique element $\chi_{\pi} \in P(\mathbb{G})$ satisfying
\[
(\varphi(\chi_{\pi}\,\cdot\,)\otimes \iota)\delta_{\rho} =
\begin{cases}
0 & \text{if $\rho \in \Irr \mathbb{G}$ and $\rho\ncong \pi$}, \\
\iota_{\mathscr{H}_{\rho}} &\text{if $\rho \cong \pi$}.
\end{cases}
\]
We will also use the shorthand notation \[\omega_{\pi}=\hat{\chi}_{\pi}=\varphi(\chi_{\pi}\,\cdot\,)\in C(\mathbb{G})^*.\] For an arbitrary representation $\pi$, we write $\chi_{\pi}$
for the sum of the quantum characters of those irreducible representations which appear in $\pi$ with non-zero multiplicity.
\end{Def}

We note that $\chi_{\pi}$ is well-defined by the Peter-Weyl theory for compact quantum groups. We record the following facts about $\chi_{\pi}$.

\begin{Lem}\label{LemChar} For each representation $\pi$ of $\mathbb{G}$, we have that \[ S^2(\chi_{\pi})=\chi_{\pi}, \;S(\chi_{\pi})^* = \chi_{\pi}, \textrm{ and }\chi_{\pi}^* \in
C(\mathbb{G})_{\pi}.\]\end{Lem}

Consider now the bounded map \[E_{\pi}\colon A\rightarrow A\colon a\rightarrow (\iota\otimes \omega_{\pi})\alpha(a).\] The following lemma is a reformulation of (part
of)~\cite[Theorem 1.5]{Pod1} (which holds regardless of any unitality assumption on $A$).

\begin{Lem} The map $E_{\pi}$ is an idempotent onto $A_{\pi}$.\end{Lem}

The following Pimsner-Popa type inequality will be crucial in what follows.

\begin{Lem}\label{LemBoun} For each finite dimensional unitary representation $\pi$ of $\mathbb{G}$, there exists $c_{\pi}>0$ such that for all $n\in \mathbb{N}_0$ and $a\in
M_n(\mathbb{C})\otimes A$, we have \[(\iota \otimes E_{\pi})(a)^*(\iota\otimes E_{\pi})(a) \leq c_{\pi}^2 (\iota\otimes E_B)(a^*a).\]\end{Lem}

\begin{proof} Since the estimating constant depends only on the representation and not on the coaction, we can restrict to the case $n=1$, by replacing the coaction $\alpha$ with the
amplified coaction $\iota\otimes \alpha$.

Then, using that $\iota\otimes \varphi$ is a c.c.p.~ map of $A\omin C(\mathbb{G})$ onto $A\otimes \mathbb{C}$, and using the inequality $\phi(x)^*\phi(x)\leq \phi(x^*x)$ for a general
c.c.p.~ map $\phi$, we find \begin{equation*} \begin{split} E_{\pi}(a)^*E_{\pi}(a) &= (\iota\otimes \varphi)((1\otimes \chi_{\pi})(\alpha(a)))^* (\iota\otimes \varphi)((1\otimes
\chi_{\pi})\alpha(a))\\ &\leq (\iota\otimes \varphi)(((1\otimes \chi_{\pi})\alpha(a))^*((1\otimes \chi_{\pi})\alpha(a)))\end{split}\end{equation*}
The right hand side is bounded from above by $\|\chi_{\pi}\|^2 (\iota\otimes \varphi)(\alpha(a^*a))= \|\chi_{\pi}\|^2 E_B(a^*a)$ by the positivity of $\iota \otimes \phi$. Setting
$c_\pi = \| \chi_\pi \|$, we obtain the assertion.
\end{proof}

\begin{Cor}\label{CorPropHilb} \begin{thmenum}\item The right $B$-module $A_{\pi}$ is complete with respect to its $B$-valued inner product, that is, $A_{\pi}$ is a right Hilbert
$B$-bimodule.
\item We have $\lbrack A_{\pi}\cdot B \rbrack = A_{\pi}$ and $\lbrack A\cdot B\rbrack = A$, where the closure is with respect to the C$^*$-norm.
\item For each representation $\pi$, the space $A\smbox \mathscr{H}_{\pi}$ is a right Hilbert $B$-bimodule.
\end{thmenum}
\end{Cor}

By a right Hilbert $B$-bimodule, we mean a right Hilbert $B$-module $\mathcal{E}$ together with a non-degenerate $^*$-representation of $B$ as adjointable operators on $\mathcal{E}$.

\begin{proof} By the previous lemmas, we have the Pimsner-Popa type inequalities \[ \|\langle a,a\rangle_B\|^{1/2} \leq \|a\| \leq c_{\pi} \|\langle a,a\rangle_B\|^{1/2}\] for $a\in
A_{\pi}$. As $A_{\pi}$ is closed in the C$^*$-algebra norm (being the image of a norm-bounded projection), this proves the first part of the corollary.

One also easily shows that if $b_i$ is an approximate unit for $B$, then for each $a\in A$, we have $ab_i \rightarrow a$ in the Hilbert module norm, and hence also in the C$^*$-norm.
Since the linear span $P(A)$ of all the isotypical components is norm-dense in $A$ by~\cite[Theorem 1.5]{Pod1}, which holds regardless of the unitality of $A$, the second part of the
corollary is proven.

The third part then follows for $\pi$ irreducible because $A\smbox \mathscr{H}_{\pi}$ is an orthogonally complemented summand of $A_{\pi}$ by Lemma~\ref{LemMult}. The general case
follows since $A\smbox -$ preserves finite direct sums.
\end{proof}

\begin{Not} To distinguish the two different norms on $A_{\pi}$, we will write \begin{equation*}\begin{split}
A_{\pi}&\colon A_{\pi} \text{ endowed with the restriction of the C$^*$-norm of $A$},\\
\mathcal{A}_{\pi}&\colon A_{\pi} \text{ endowed with the right Hilbert $B$-module structure}.
\end{split}\end{equation*}
The natural identification map $A_{\pi}\rightarrow \mathcal{A}_{\pi}$ will be denoted by $\Lambda_{\mathcal{A}}\colon A_{\pi}\rightarrow \mathcal{A}_{\pi}$.

We will also denote by $\mathcal{A}$ the right Hilbert $B$-module completion of $A$ with respect to $\langle \,\cdot\,,\,\cdot\,\rangle_B$. We then have natural inclusions
$\mathcal{A}_{\pi}\subseteq \mathcal{A}$.
\end{Not}

We can also put a \emph{right} Hilbert $B$-$\mathcal{K}(\mathcal{A}_{\overline{\pi}})$-bimodule structure on each $\mathcal{A}_{\pi}$ by means of the ordinary left $B$-module
structure and the $\mathcal{K}(\mathcal{A}_{\overline{\pi}})$-valued inner product \[\langle x,y\rangle_{\mathcal{K}(\mathcal{A}_{\overline{\pi}})} =
\Lambda_{\mathcal{A}}(x^*)\Lambda_{\mathcal{A}}(y^*)^*.\] Then we have $\|\langle x,x\rangle_{\mathcal{K}(\mathcal{A}_{\overline{\pi}})}\| = \|E_B(xx^*)\|$, which is the norm squared
obtained by considering $A_{\pi}$ as a \emph{left} Hilbert $B$-module by means of the inner product $_B\langle x,y \rangle= E_B(xy^*)$. Hence we can extend our previous notation as
follows.

\begin{Not} We write \begin{equation*}\mathscr{A}_{\pi}\colon A_{\pi} \text{ endowed with the right $\mathcal{K}(A_{\overline{\pi}})$-Hilbert module structure}.\end{equation*} We
denote the natural identification map $A_{\pi}\rightarrow \mathscr{A}_{\pi}$ by $\Lambda_{\mathscr{A}}\colon A_{\pi} \rightarrow \mathscr{A}_{\pi}$. In the same way, one defines the
right Hilbert $\mathcal{K}(\mathcal{A})$-module completion of $A$ by $\mathscr{A}$.
\end{Not}

Note that we have $\| x \|_{\mathscr{A}_\pi} = \| x^* \|_{\mathcal{A}_{\overline{\pi}}}$.

\begin{Lem}\label{LemNormCompar}
The maps $\Lambda_{\mathcal{A}}$, $\Lambda_{\mathscr{A}}$ are completely bounded.  Similarly, when $\rho$ is a finite dimensional unitary representation of $\mathbb{G}$, the identity
map on $A_\rho$ is completely bounded with respect to the $\mathcal{A}_\rho$-norm or $\mathscr{A}_\rho$-norm on the domain and the $A_\rho$-norm on the codomain.
\end{Lem}

\begin{proof}
The boundedness for the first two maps follows from a standard calculation using the fact that $E_B$ is c.c.p.  Furthermore, Lemma~\ref{LemBoun} shows that $\iota$ is completely
bounded for the $\mathcal{A}_\rho$-norm on the domain, with a norm bounded from above by $c_\pi$.  By the above remark and the general equality $\| x^* x \| = \| x x^* \|$ for
C$^*$-norms, we obtain an analogous cb-norm estimate of $\iota$ from above by $c_{\overline{\pi}}$ for the $\mathscr{A}_{\pi}$-norm on the domain.
\end{proof}

\subsection{Crossed product and its corners}

Let us put $\hat{h} = \phi(h \cdot) \in C(\mathbb{G})^*$ for $h \in P(\mathbb{G})$, and consider the space
\[
\widehat{P(\mathbb{G})} = \{\hat{h} \mid h\in P(\mathbb{G})\} \subseteq C(\mathbb{G})^*.
\]
It is a (generally non-unital) $^*$-algebra by the convolution product and the $^*$-operation which is determined by \[\omega^*(x) = \overline{\omega(S(x)^*)},\qquad x\in
P(\mathbb{G}).\] This $^*$-algebra admits a (non-unital) universal C$^*$-envelope $C^*(\mathbb{G})$, which is a C$^*$-algebraic direct sum of matrix algebras, the components of which
are labeled by $\Irr \mathbb{G}$.

\noindent It is easily seen that we can extend $\alpha$ to a unitary comodule structure on $\mathcal{A}$, \[\alpha_{\mathcal{A}}\colon \mathcal{A}\rightarrow \mathcal{A}\omin
C(\mathbb{G}),\] where we view the right hand side as a right Hilbert $B\omin C(\mathbb{G})$-module in a natural way. Then to the coaction $\alpha_{\mathcal{A}}$, we can associate the
(non-degenerate) $^*$-representation \[\widehat{\pi}_{\alpha}\colon C^*(\mathbb{G}) \rightarrow \mathcal{L}(\mathcal{A}),\] where $\mathcal{L}(\mathcal{A})$ denotes the space of
adjointable linear endomorphisms of $\mathcal{A}$. The map $\widehat{\pi}_{\alpha}$ is uniquely determined by \[\widehat{\pi}_{\alpha}(\omega)\Lambda_{\mathcal{A}}(x) =
\Lambda_{\mathcal{A}}((\iota\otimes \omega)\alpha(x))\qquad \textrm{for }x\in A,\omega\in \widehat{P(\mathbb{G})}.\] Moreover, we have that \[p_{\pi} =
\widehat{\pi}_{\alpha}(\widehat{\chi}_{\pi})\] is the projection of $\mathcal{A}$ onto $\mathcal{A}_{\pi}$. In particular, $\mathcal{A}_{\pi}$ is complemented in $\mathcal{A}$.

If we perform this construction for the particular case of $(A,\alpha) = (C(\mathbb{G}),\Delta)$, the space $\mathcal{A}$ becomes a Hilbert space, which we denote by
$\mathscr{L}^2(\mathbb{G})$. It is the completion of $C(\mathbb{G})$ with respect to the inner product $\langle x,y \rangle = \varphi(x^*y)$. In this case, we denote the associated
GNS-map by $\Lambda_{\varphi}\colon C(\mathbb{G})\rightarrow \mathscr{L}^2(\mathbb{G})$, and we also write $\mathscr{L}^2(\mathbb{G})_{\pi} = \Lambda_{\varphi}(C(\mathbb{G})_{\pi})$
for the finite dimensional Hilbert space of the matrix coefficients for $\pi$. The associated representation $\widehat{\pi}_{\Delta}$ is then faithful, and we will in the following
treat $\widehat{\pi}_{\Delta}$ as the identity map, so that $C^*(\mathbb{G})\subseteq B(\mathscr{L}^2(\mathbb{G}))$.

Consider now the right Hilbert $B$-module $\mathcal{A}\otimes \mathscr{L}^2(\mathbb{G})$. It carries a natural non-degenerate $^*$-representation of $A\omin C(\mathbb{G})$ as
$B$-endomorphisms (cf.~\cite{Lan1}, page 34). By means of the homomorphism $\alpha$, we obtain a representation of $A$ on $\mathcal{A} \otimes \mathscr{L}^2(\mathbb{G})$ as well. Note
that this representation might not be faithful as we do not assume $C(\mathbb{G})$ to be reduced. The space $\mathcal{A} \otimes \mathscr{L}^2(\mathbb{G})$ also carries a
$^*$-representation of $C^*(\mathbb{G})$, acting by the ordinary convolution on the second leg.

The crossed product of $A$ by $\mathbb{G}$ for the action $\alpha$ is defined as
\[
A\rtimes \mathbb{G} = \lbrack \alpha(A)(1\otimes C^*(\mathbb{G}))\rbrack \subset \mathcal{L}(\mathcal{A} \otimes \mathscr{L}^2(\mathbb{G})),
\]
which will be a C$^*$-algebra. One could also use the right Hilbert $A$-module $A\otimes \mathscr{L}^2(\mathbb{G})$ in the above construction, and this would give the same
C$^*$-algebra. We define \[(A\rtimes \mathbb{G})_{\pi} =  \lbrack(A\rtimes \mathbb{G})(1\otimes \widehat{\chi}_{\pi})\rbrack,\] which is a closed left ideal in $A\rtimes \mathbb{G}$.
We similarly define \[_{\pi}(A\rtimes \mathbb{G}) = (A\rtimes \mathbb{G})_{\pi}^*, \quad _{\rho}(A\rtimes \mathbb{G})_{\pi} = \lbrack (A\rtimes \mathbb{G})_{\rho}^* (A\rtimes
\mathbb{G})_{\pi}\rbrack.\] The $_{\pi}(A\rtimes \mathbb{G})_{\pi}$ are then C$^*$-algebras.

We state a lemma which lets us realize our Hilbert modules $\mathcal{A}$ and $\mathscr{A}$ in terms of this crossed product. We omit the proof which follows from a straightforward
computation. In the lemma, we will denote the unit of $C(\mathbb{G})$ by $\chi_{\mathrm{triv}}$ for clarity.

\begin{Lem}\label{LemIdeal} There exists a unique right Hilbert $B$-module structure on $(A\rtimes \mathbb{G})_{\mathrm{triv}}$ such that \[\langle x,y\rangle_B\otimes
\widehat{\chi}_{\mathrm{triv}} = x^*y,\qquad \textrm{for all }x,y\in (A\rtimes \mathbb{G})_{\mathrm{triv}},\] and we have then a natural isomorphism \[\mathcal{A}\rightarrow (A\rtimes
\mathbb{G})_{\mathrm{triv}}\colon \Lambda_{\mathcal{A}}(a) \rightarrow \alpha(a)(1\otimes \widehat{\chi}_{\mathrm{triv}}),\] of right Hilbert $B$-bimodules.
\end{Lem}

By the previous lemma, it follows that we can realize $\mathcal{K}(\mathcal{A})$ as a closed 2-sided ideal inside $A\rtimes \mathbb{G}$ by means of a $^*$-homomorphism
\begin{equation}\label{EqPiAlphaDef}
\Pi_{\alpha}\colon \mathcal{K}(\mathcal{A}) \rightarrow A\rtimes \mathbb{G}\colon \Lambda_{\mathcal{A}}(x)\Lambda_{\mathcal{A}}(y)^* \rightarrow \alpha(x)(1\otimes
\widehat{\chi}_{\mathrm{triv}})\alpha(y^*),
\end{equation}
which is in general degenerate. We can interpret $\mathscr{A}$ as a right $B$-$A\rtimes \mathbb{G}$-Hilbert bimodule by composing its $\langle
\,\cdot\,,\,\cdot\,\rangle_{\mathcal{K}(\mathcal{A})}$-valued inner product with the map $\Pi_{\alpha}$. We then have an isomorphism of right Hilbert $B$-$A\rtimes
\mathbb{G}$-bimodules \[\mathscr{A}\rightarrow \,\!_{\mathrm{triv}}(A\rtimes \mathbb{G})\colon \Lambda_{\mathscr{A}}(a)\rightarrow (1\otimes
\widehat{\chi}_{\mathrm{triv}})\alpha(a).\]

Note now that we can identify $\mathcal{A}$ with the closure of $(\Lambda_{\mathcal{A}}\otimes \Lambda_{\varphi})(\alpha(A))$ inside $\mathcal{A}\otimes \mathscr{L}^2(\mathbb{G})$. We
obtain in this way that $\mathcal{A}$ is an $A\rtimes \mathbb{G}$-invariant subspace of $\mathcal{A}\otimes \mathscr{L}^2(\mathbb{G})$, and we denote the resulting $^*$-representation
of $A\rtimes \mathbb{G}$ on $\mathcal{A}$ by $\pi_{\textrm{red}}$.  Its restriction to $A$ is given by left multiplication, while its restriction to $C^*(\mathbb{G})$ is the
representation $\widehat{\pi}_{\alpha}$. It is also easy to see that $\pi_{\textrm{red}} \circ \Pi_{\alpha}$ is the identity map on $\mathcal{K}(\mathcal{A})$.

\subsection{Galois maps}

We can interpret $A$ itself as a right Hilbert $B$-$A$-Hilbert module, the bimodule structure being given by multiplication and the $A$-valued inner product by $\langle x,y\rangle_A =
x^*y$. Further, for any representation $\pi$, we interpret $\mathscr{A}\otimes \mathscr{L}^2(\mathbb{G})_{\pi}$ as a right $A\rtimes \mathbb{G}$\,-Hilbert module by means of the
identification map
\begin{equation}\label{EqScrAL2GpiIsPiAG}
\mathscr{A}\otimes \mathscr{L}^2(\mathbb{G})_{\pi} \rightarrow \,\!_{\pi}(A\rtimes \mathbb{G})\colon \Lambda_{\mathscr{A}}(x) \otimes \Lambda_{\varphi}(S(h)) \rightarrow (1\otimes
\widehat{h})\alpha(x),\qquad x\in A,h\in C(\mathbb{G})_{\overline{\pi}}.
\end{equation}
It requires a small argument to show that this Hilbert module is complete, but we will actually never use this fact.

\begin{Def-Prop}\label{PropGalIsom} The map \[A_{\pi} \algtensor P(A) \rightarrow P(A)\otimes C(\mathbb{G})_{\pi}\colon a\otimes a' \rightarrow \alpha(a)(a'\otimes 1)\] extends
uniquely to isometric maps \[G_{\pi}\colon \mathcal{A}_{\pi}\otimes_B A \rightarrow A\otimes \mathscr{L}^2(\mathbb{G})_{\pi},\; \mathcal{G}_{\pi}\colon \mathcal{A}_{\pi}\otimes_B
\mathcal{A} \rightarrow \mathcal{A}\otimes \mathscr{L}^2(\mathbb{G})_{\pi}\; \textrm{ and }\; \mathscr{G}_{\pi}\colon \mathcal{A}_{\pi}\otimes_B \mathscr{A} \rightarrow
\mathscr{A}\otimes \mathscr{L}^2(\mathbb{G})_{\pi} \] between right Hilbert modules.

We shall call the above maps the $\pi$-localised \emph{C$^*$-}, \emph{Hilbert C$^*$-} and \emph{crossed product-Galois map}, respectively.
\end{Def-Prop}

Here $-\otimes_B-$ denotes the interior tensor product over $B$, see~\cite{Lan1}. Also, for the first two maps, the tensor product on the right is simply an amplification of the
corresponding right Hilbert modules with the Hilbert space $\mathscr{L}^2(\mathbb{G})_{\pi}$.

\begin{proof} A trivial computation shows that the proposed formula for $G_{\pi}$ and $\mathcal{G}_{\pi}$ respects the $A$-valued, resp.~ $B$-valued inner product on elementary
tensors, so that they descend and complete to maps with domain $\mathcal{A}_{\pi}\otimes_B A$, resp.~ $\mathcal{A}_{\pi}\otimes_B \mathcal{A}$. The statement about $\mathscr{G}_{\pi}$
follows from the following easily verified identity inside $A\rtimes \mathbb{G}$: \[(1 \otimes \widehat{S^{-1}(x_{(1)})})\alpha(x_{(0)}) = \alpha(x)(1\otimes
\widehat{\chi}_{\mathrm{triv}}), \qquad x\in A_{\pi},\] where we used the Sweedler notation $\alpha(x) = x_{(0)}\otimes x_{(1)}$.
\end{proof}

One can show that these Galois maps add up to respective isometries $\mathcal{A}\otimes_B- \rightarrow -\otimes \mathscr{L}^2(\mathbb{G})$ which are $C^*(\mathbb{G})$-equivariant from
the second to the first leg - we will use this observation only for the cases $A$ and $\mathcal{A}$. Note that the terminology `Galois map' comes from the corresponding Hopf algebraic
theory, cf.~\cite{Sch1}.

\section{Adjointability of the Galois maps}\label{adjt}

\noindent We keep the notational conventions of the previous section. In particular, $\mathbb{G}$ is a compact quantum group acting on a not necessarily unital C$^*$-algebra $A$, with
fixed point C$^*$-algebra $B$. We first recall a general fact about the amplification of morphisms of Hilbert C$^*$-bimodules.

\begin{Lem}\label{LemModMapAmpl}
Let $C$ and $D$ be C$^*$-algebras, $\mathcal{E}$ and $\mathcal{E}'$ respectively be a right Hilbert C$^*$-$B$-$C$-bimodule and a right Hilbert C$^*$-$B$-$D$-bimodule.  If $T$ is a
completely bounded $B$-module map from $\mathcal{E}$ to $\mathcal{E}'$, then $\iota \otimes T$ descends to a bounded map $\iota \otimes_B T$ of norm at most $\| T \|_{\mathrm{cb}}$
from $\mathcal{F} \otimes_B \mathcal{E}$ to $\mathcal{F} \otimes_B \mathcal{E}'$ for any right Hilbert C$^*$-$B$-module $\mathcal{F}$.
\end{Lem}

\begin{proof}
This is shown by a standard argument.  Let $x_1, \ldots, x_n$ be elements in $\mathcal{E}$, and $y_1, \ldots, y_n$ be in $\mathcal{F}$.  Then the matrix $Y = (\langle y_i, y_j
\rangle_B)_{i, j}$ in $M_n(B)$ is a positive element.  Hence there exists $b = (b_{i, j})_{i,j} \in M_n(B)$ satisfying $b^* b = Y$.  Then one has
\[
\Bigl\langle \iota \otimes T\Bigl(\sum_i y_i \otimes x_i \Bigr), \iota \otimes T \Bigl(\sum_j y_j \otimes x_j \Bigr) \Bigr\rangle_D = \sum_{i, j} \left\langle T(x_i), Y_{i, j} T(x_j)
\right\rangle_D.
\]
The right hand side is equal to
\[
\sum_{i, j, k} \left\langle T(x_i), b_{k, i}^* b_{k, j} T(x_j) \right\rangle_D = \| \iota \otimes T (\xi) \|^2,
\]
where $\xi \in \mathbb{C}^n\otimes \mathcal{E}$ is a column vector whose $k$-th component is equal to $\sum_i b_{k, i} x_i$.  By the complete boundedness of $T$, we obtain
\[
\| \iota \otimes T (\xi) \|^2 \le \| T \|_{\mathrm{cb}}^2 \| \xi \|^2 = \| T \|_{\mathrm{cb}}^2 \Bigl \| \Bigl\langle \sum_i y_i \otimes x_i, \sum_j y_j \otimes x_j \Bigr\rangle_C
\Bigr \|^2,
\]
which implies the desired estimate $\| \iota \otimes_B T \| \le \| T \|_{\mathrm{cb}}$.
\end{proof}

The adjointability of the Galois maps are in fact equivalent conditions.

\begin{Prop}\label{PropAd} Let $\mathbb{G}$ be a compact quantum group acting on a unital C$^*$-algebra $A$. Let $\pi$ be a finite dimensional unitary representation of $\mathbb{G}$.
The following conditions are equivalent.
\begin{thmenum} \item The $\pi$-localized Hilbert-C$^*$-Galois map $\mathcal{G}_{\pi}$ has an adjoint.
\item The $\pi$-localized C$^*$-Galois map $G_{\pi}$ has an adjoint.
\item The $\pi$-localized crossed product-Galois map $\mathscr{G}_{\pi}$ has an adjoint.
\end{thmenum}
\end{Prop}

\begin{proof}
Each of the implications can be argued along the same pattern; given the adjoint of one of the Galois maps, we may restrict it to isotypical components $A_\rho \algtensor
C(\mathbb{G})_\pi$ for $\rho \in \Irr \mathbb{G}$.  Then, Lemmas~\ref{LemNormCompar} and~\ref{LemModMapAmpl} imply that those restrictions are continuous with respect to the other
norms.  The resulting map on $P(A) \algtensor C(\mathbb{G})_\pi$ is shown to be a formal adjoint of the Galois with respect to the corresponding algebra valued inner product.  By
inner product characterization of duals for Hilbert C$^*$-modules, we obtain that this formal adjoint extends to the actual adjoint.  As an illustration, let us prove the implication
$(1) \Rightarrow (2)$.

Assume that (1) holds.

Take a finite dimensional unitary representation $\rho$ of $\mathbb{G}$, $h\in C(\mathbb{G})_{\pi}$, and $a\in A_{\rho}$. We first claim that
$\mathcal{G}_{\pi}^*(\Lambda_{\mathcal{A}}(a)\otimes \Lambda_{\varphi}(h))$ is contained in $\mathcal{A}_{\pi}\otimes_B \mathcal{A}_{\overline{\pi}\times \rho}$. Observe that we can
write
\[
\mathcal{A}_{\pi} \otimes_B \mathcal{A} = \bigoplus_{\theta \in \Irr \mathbb{G}}( \mathcal{A}_{\pi} \otimes_B \mathcal{A}_{\theta}).
\]
Take an irreducible representation $\theta$ which does not appear in $\overline{\pi}\times \rho$, and take $x\in A_{\pi}$, $y\in A_{\theta}$. By a direct calculation, we obtain
\begin{multline*} \langle \mathcal{G}_{\pi}^*(\Lambda_{\mathcal{A}}(a)\otimes \Lambda_{\varphi}(h)), \Lambda_{\mathcal{A}}(x)\otimes \Lambda_{\mathcal{A}}(y)\rangle  = (\iota\otimes
\varphi\otimes \varphi)((\alpha(a^*)\otimes h^*)((\alpha\otimes \iota)\alpha(x))(\alpha(y)\otimes 1)).\end{multline*} But we see that the second factor of the evaluated element only
contains matrix coefficients of the representation $\overline{\rho}\times \pi\times \theta$, which does not contain the trivial representation by the assumption on $\theta$ and
Frobenius reciprocity. Hence the above expression is zero, and the claim follows.

Combining Lemmas~\ref{LemNormCompar} and~\ref{LemModMapAmpl}, we obtain a map $\mathcal{A}_\pi \otimes_B \mathcal{A}_{\overline{\pi} \times \rho} \rightarrow \mathcal{A}_\pi \otimes_B
A_{\overline{\pi} \times \rho}$.  Composing this with $\mathcal{G}_\pi^*$, we obtain a map from $A_\rho \algtensor C(\mathbb{G})_\pi$ to $\mathcal{A}_\pi \otimes_B A_{\overline{\pi}
\times \rho}$.  Taking linear combinations, we obtain a map $F_0$ from $P(A) \algtensor C(\mathbb{G})_\pi$ to $\mathscr{A}_\pi \otimes_B A$. We want to show that $F_0$ is bounded, and
that its closure equals $G_{\pi}^*$.

Fix $a\in A_{\rho}$ and $h\in C(\mathbb{G})_{\pi}$, and take elements $(x_{n,i})_{i = 1}^{N_n} \in A_{\pi}$ and $(y_{n,i})_{i = 1}^{N_n} \in A_{\overline{\pi}\times \rho}$ for $n, N_n
\in \mathbb{N}$ such that
\[
\mathcal{G}_{\pi}^*(\Lambda_{\mathcal{A}}(a)\otimes \Lambda_{\varphi}(h)) = \lim_{n\rightarrow \infty} \sum_{i=1}^{N_n} \Lambda_{\mathcal{A}}(x_{n,i})\otimes
\Lambda_{\mathcal{A}}(y_{n,i}).
\]
Take $x\in A_{\pi}$ and $y\in P(A)$.  On the one hand, we have the convergence
\begin{multline*}
\sum_i E_B(y_{n,i}^*E_B(x_{n,i}^*x)y) = \sum_{i} \langle \Lambda_\mathcal{A}(x_{n, i}) \otimes \Lambda_\mathcal{A}(y_{n, i}), \Lambda_\mathcal{A}(x) \otimes \Lambda_\mathcal{A}(y)
\rangle\\
\underset{n \rightarrow \infty}\longrightarrow \langle \Lambda_\mathcal{A}(a) \otimes \Lambda_\varphi(h), \mathcal{G}_\pi(x \otimes y) \rangle = (E_B\otimes \varphi)((a^*\otimes
h^*)\alpha(x)(y\otimes 1)).
\end{multline*}
On the other hand, we also have the convergence
\[
\langle F_0(a\otimes h), \Lambda_{\mathcal{A}}(x)\otimes y\rangle = \lim_{n \rightarrow \infty} \sum_{i=1}^{N_n} y_{n,i}^*E_B(x_{n,i}^*x)y.
\]
Multiplying to the right with an arbitrary $z\in P(A)$, we find that
\begin{equation*}
E_B(\langle F_0(a\otimes h),\Lambda_{\mathcal{A}}(x)\otimes y)\rangle \cdot z ) \\ = E_B(\langle \Lambda_{\mathcal{A}}(a)\otimes
\Lambda_{\varphi}(h),G_{pi}(\Lambda_{\mathcal{A}}(x)\otimes y)\rangle \cdot z).
\end{equation*}
As $z$ was arbitrary, and as $\langle \,\cdot\,,\,\cdot\,\rangle_B$ is non-degenerate on $P(A)$, we find that
\[\langle F_0(a\otimes h),\Lambda_{\mathcal{A}}(x)\otimes y\rangle = \langle \Lambda_{\mathcal{A}}(a)\otimes \Lambda_{\varphi}(h),G_{\pi}(\Lambda_{\mathcal{A}}(x)\otimes y)\rangle.\]
From this formula, we obtain that $F_0$ descends to a contractive map $F$ from $A\otimes \mathscr{L}^2(\mathbb{G})_{\pi}$ to $\mathcal{A}\otimes_B A$, and that $F$ is then precisely
the adjoint of $G_{\pi}$, hence we obtain (2).
\end{proof}

The elements in the algebra $B$ act on $\mathcal{A}_\pi$ as left multiplication operators, which are adjointable endomorphisms for the right Hilbert $B$-module structure.  We let
$\pi_L$ denote the associated embedding of $B$ into $\mathcal{L}(\mathcal{A}_\pi)_B$.  The amplification of left multiplication defines an analogous action of $B$ on $A \smbox
\mathscr{H}_\pi$.  By abuse of notation, we  denote this representation also by $\pi_L$ (see Lemma~\ref{LemMult}).

\begin{Theorem}\label{ThmAdjProj}
Let $\alpha$ be an action of $\mathbb{G}$ on a C$^*$-algebra $A$, and let $\pi$ be a finite dimensional unitary representation of $\mathbb{G}$.  Then the following conditions are
equivalent.

\begin{thmenum}
\item The $\pi$-localized Galois maps are adjointable.
\item The image of $\,\!_\pi(A\rtimes \mathbb{G})$ under $\pi_{\textrm{red}}$ lies in $\mathcal{K}(\mathcal{A},\mathcal{A}_{\pi})$.
\item The image of $\pi_L\colon B \rightarrow \mathcal{L}(\mathcal{A}_\pi)$ is contained in $\mathcal{K}(\mathcal{A}_{\pi})$.
\item[{\rm (3')}] The image of $\pi_L\colon B \rightarrow \mathcal{L}(A\smbox \mathscr{H}_{\pi})$ is contained in $\mathcal{K}(A\smbox \mathscr{H}_{\pi})$.
\end{thmenum}

Furthermore, when $A$ is unital, these conditions are equivalent to the following statements.

\begin{enumerate}\item[{\rm (4)}] $\mathcal{A}_{\pi}$ is finitely generated and projective as a right $B$-module.
\item[{\rm (4')}] $\mathcal{A}\smbox \mathscr{H}_{\pi}$ is finitely generated projective as a right $B$-module.
\end{enumerate}
\end{Theorem}

\begin{proof}
\emph{Proof of $(1)\Rightarrow (2)$}. Let us identify $\mathscr{A}\otimes \mathscr{L}^2(\mathbb{G})_{\pi}$ with $\,\!_{\pi}(A\rtimes \mathbb{G})$ as explained above
Proposition~\ref{PropGalIsom}. Let us also identify $\mathcal{A}_{\pi}\otimes_B \mathscr{A}$ with the right Hilbert $\mathcal{K}(\mathcal{A})$-module
$\mathcal{K}(\mathcal{A},\mathcal{A}_{\pi})$ of compact operators from $\mathcal{A}$ to $\mathcal{A}_{\pi}$, by means of the natural map
\begin{equation*}
\Upsilon_\pi\colon \mathcal{A}_{\pi}\otimes_B \mathscr{A} \rightarrow  \mathcal{K}(\mathcal{A},\mathcal{A}_{\pi})\colon \Lambda_{\mathcal{A}}(x)\otimes \Lambda_{\mathscr{A}}(y)
\rightarrow \Lambda_{\mathcal{A}}(x)\Lambda_{\mathcal{A}}(y^*)^* .
\end{equation*}
Then $\mathscr{G}_{\pi}$ becomes the map $\Pi_{\alpha}$ of~\eqref{EqPiAlphaDef}.

As $\mathscr{G}_{\pi}$ is adjointable, we have, for $x$ in ${}_{\pi}(A\rtimes \mathbb{G}) \simeq \mathscr{A} \otimes \mathscr{L}^2(\mathbb{G})_\pi$ (see~\eqref{EqScrAL2GpiIsPiAG}) and
$y$ in $\mathcal{K}(\mathcal{A},\mathcal{A}_{\pi})$, that \[\Pi_{\alpha}((\mathscr{G}_{\pi}^*(x))^* y)  = x^*\Pi_{\alpha}(y).\] Applying $\pi_{\textrm{red}}$, we conclude that
$(\mathscr{G}_{\pi}^*(x))^* y = \pi_{\textrm{red}}(x)^*y$ for all $y\in \mathcal{K}(\mathcal{A},\mathcal{A}_{\pi})$, and hence $\mathscr{G}_{\pi}^*(x)= \pi_{\textrm{red}}(x) \in
\mathcal{K}(\mathcal{A})$.

The implication $(2)\Rightarrow (3)$ is of course trivial.

\emph{Proof of $(3)\Rightarrow (1)$}. Take $a\in P(A)$, say $a\in A_{\rho}$ for some representation $\rho$, and $h\in C(\mathbb{G})_{\pi}$. Since $\lbrack A_{\rho}B\rbrack = A_{\rho}$
by Corollary~\ref{CorPropHilb}, the operator $\widehat{\pi}_{\alpha}(\widehat{S^{-1}(h)})a$ is a compact operator in $\mathcal{K}(\mathcal{A}_{\overline{\rho}\times
\pi},\mathcal{A}_{\pi})$.

Let $\Upsilon_{\pi,\overline{\pi}\times \rho}\colon \mathcal{A}_{\pi} \otimes_B \mathscr{A}_{\overline{\pi}\times \rho} \rightarrow
\mathcal{K}(\mathcal{A}_{\overline{\rho}\times\pi},\mathcal{A}_{\pi})$ be the restriction of $\Upsilon_\pi$.  It is a complete isometry.  Moreover, combining
Lemmas~\ref{LemNormCompar} and~\ref{LemModMapAmpl}, we obtain an injective bounded map $\mathscr{S}_{\pi, \overline{\pi} \times \rho}$ from $\mathcal{A}_\pi \otimes_B
\mathscr{A}_{\overline{\pi} \times \rho}$ to $\mathcal{A}_\pi \otimes_B \mathcal{A}_{\overline{\pi} \times \rho}$.

Define a map $F_0$ from $P(A) \algtensor C(\mathbb{G})_{\pi}$ to $\mathcal{A}_{\pi}\otimes_B\mathcal{A}$ by
\[
F_0(a \otimes h) = (\mathscr{S}_{\pi,\overline{\pi}\times\rho}\circ \Upsilon_{\pi,\overline{\pi}\times \rho}^{-1})(\widehat{\pi}_{\alpha}(\widehat{S^{-1}(h)})a).
\]
Let us choose for each $n\in \mathbb{N}$ a finite collection of $x_{n,i}\in A_{\pi},y_{n,i}\in A_{\overline{\pi}\times \rho}$ such that
\[
F_0(a \otimes h) = \lim_{n\rightarrow \infty} \sum_i \Lambda_{\mathcal{A}}(x_{n,i})\otimes \Lambda_{\mathcal{A}}(y_{n,i}).
\]
Then by definition of $F_0(a \otimes h)$, we have the equality \[\lim_n \sum_i \Lambda_{\mathcal{A}}(x_{n,i})\Lambda_{\mathcal{A}}(y_{n,i}^*)^*
=\widehat{\pi}_{\alpha}(\widehat{S^{-1}(h)})a\] as operators from $\mathcal{A}_{\overline{\rho}\times \pi}$ to $\mathcal{A}_{\pi}$. Applying the $^*$-operation to both sides, and
applying these expressions to $\Lambda_{\mathcal{A}}(x)$ for some $x\in A_{\pi}$, we see that \[\lim_{n} \sum_i \Lambda_{\mathcal{A}}(y_{n,i}^*E_B(x_{n,i}^*x)) =\Lambda_{\mathcal{A}}(
(\iota\otimes \varphi)((a^*\otimes h^*)\alpha(x))).\]

In the above formula, the vectors inside $\Lambda_{\mathcal{A}}$ belong to $\mathcal{A}_{\overline{\pi} \times \rho}$.  Since the norms on $A_{\overline{\pi} \times \rho}$ and
$\mathcal{A}_{\overline{\pi} \times \rho}$ are equivalent, the convergence still holds in $A_{\overline{\pi} \times \rho}$. Hence if $x\in A_{\pi}$ and $y\in P(A)$, we have
\begin{eqnarray*}
\langle F_0(a \otimes h), \Lambda_{\mathcal{A}}(x)\otimes \Lambda_{\mathcal{A}}(y)\rangle &=& \lim_n \sum_i E_B(y_{n,i}^* E_B(x_{n,i}^*x)y) \\ &=& E_B(a^* (\iota\otimes
\varphi(h^*\,\cdot\,))(\alpha(x))y)\\ &=& \langle \Lambda_{\mathcal{A}}(a)\otimes \Lambda_{\varphi}(h),\mathcal{G}_{\pi}(\Lambda_{\mathcal{A}}(x)\otimes
\Lambda_{\mathcal{A}}(y))\rangle.\end{eqnarray*}

As in the proof of Proposition~\ref{PropAd}, we can conclude that $F_0$ extends to a bounded map $F$ from $\mathcal{A}\otimes \mathscr{L}^2(\mathbb{G})_{\pi}$ to
$\mathcal{A}_{\pi}\otimes_B \mathcal{A}$, which is then the adjoint of $\mathcal{G}_{\pi}$.  This way we obtain that the $\pi$-localized Galois maps are adjointable.

The equivalence between (3) and (3'), and between (4) and (4') follows from Lemma~\ref{LemMult}, as those conditions on $\pi$ are clearly equivalent to having the corresponding ones
for each irreducible subrepresentation of $\pi$.

Finally, assuming $A$ is unital, we show that (3) and (4) are equivalent. In fact, if $A$ is unital, also $B$ is unital, and the third condition simply says that
$\mathcal{K}(\mathcal{A}_{\pi})$ is unital. In particular, $\mathcal{A}_{\pi}$ must be countably generated as a Hilbert module. By~\cite[Lemma~6.5]{Kas1} and Kasparov's stabilisation
theorem~\cite[Corollary 6.3]{Lan1}, $\mathcal{A}_{\pi}\cong pB^n$ for some $n\in \mathbb{N}$ and some self-adjoint projection $p\in M_n(B)$, and is hence finitely generated
projective. Conversely, if $A_{\pi}$ is finitely generated projective,~\cite[Lemma~6.5]{Kas1} implies that $\mathcal{K}(\mathcal{A}_{\pi})$ is unital.  Since
$\mathcal{K}(\mathcal{A}_{\pi})$ is an ideal in $\mathcal{L}(\mathcal{A}_{\pi})$, this completes the proof.
\end{proof}

\begin{Rem}
\begin{enumerate}
\item The equivalence of adjointability of the Galois maps is not used in an essential way in the proof of Theorem~\ref{ThmAdjProj}.  In fact, it is possible to prove
    Proposition~\ref{PropAd} and Theorem~\ref{ThmAdjProj} `at once', by the arguments of the implications \ref{PropAd}.(3) $\Rightarrow$ \ref{ThmAdjProj}.(1) and
    \ref{ThmAdjProj}.(3) $\Rightarrow$ \ref{PropAd}.(1).
\item Let us assume our action is ergodic, which means $B= \mathbb{C}1$ (and nessarily $A$ unital). Then $\mathcal{A}$ is a Hilbert space, and as the Hilbert C$^*$-Galois map is
    then an isometry between Hilbert spaces, it is necessarily adjointable. Hence the implication $(1) \Rightarrow (5)$ of the previous result captures as a special case the fact
    that the isotypical components of an ergodic action are finite-dimensional (cf.~\cite{Boc1}). Of course, this special case can be proven more directly (whilst obtaining a
    stronger conclusion about the minimal number of generators).
\item Let $\mathbb{G}$ be a classical compact group acting on some compact space $X$. One can prove that the adjointability of the corresponding Galois maps is equivalent with the
    following purely topological condition: if $G_x\subseteq G$ is the stabilizer group of the element $x\in X$, then the assignment $x\rightarrow G_x$ is continuous, with respect
    to the natural topology known as finite or Vietoris topology on the space of closed subsets of $G$.
\end{enumerate}
\end{Rem}

We end this section with the following observation about the range projection of the Galois isometries.

\begin{Prop}\label{PropComp} Suppose that the $\pi$-localized Galois maps are adjointable. Then $\mathscr{G}_{\pi}\mathscr{G}_{\pi}^* $ is a central element inside
$\mathcal{M}(\,\!_{\pi}(A\rtimes \mathbb{G})_{\pi})$, and \[\mathscr{G}_{\pi}\mathscr{G}_{\pi}^* = G_{\pi}G_{\pi}^* = \mathcal{G}_{\pi}\mathcal{G}_{\pi}^*.\]
\end{Prop}

\begin{proof} Let us write $P_{\pi} = \mathscr{G}_{\pi}\mathscr{G}_{\pi}^*$. Then we have \[P_{\pi} \in \mathcal{L}(\mathscr{A}\otimes \mathscr{H}_{\pi})\cong
\mathcal{L}(\,\!_{\pi}(A\rtimes\mathbb{G})) \cong \mathcal{M}({}_{\pi}(A\rtimes\mathbb{G})_{\pi}),\] where $\mathcal{M}$ means taking the multiplier C$^*$-algebra. Moreover, as in the
proof of the previous theorem, we can interpret $\mathscr{G}$ and its adjoint as the maps $\Pi_{\alpha}$ and $\pi_{\textrm{red}}$ respectively. Hence for $x\in A\rtimes \mathbb{G}$,
we have $P_{\pi}x = \Pi_{\alpha}(\pi_{\textrm{red}}(x))  = \Pi_{\alpha}(\pi_{\textrm{red}}(x^*))^* = (P_{\pi}x^*)^* = xP_{\pi}$, so that $P_{\pi}$ is central.

From the `claim' in the proof of Proposition~\ref{PropAd}, we have the inclusion \[\mathcal{G}_{\pi}\mathcal{G}_{\pi}^*(\Lambda_{\mathcal{A}}(P(A))\otimes C(\mathbb{G})_{\pi})
\subseteq \Lambda_{\mathcal{A}}(P(A))\otimes C(\mathbb{G}_{\pi}).\] And from the construction of the maps $G_{\pi}^*$ and $\mathscr{G}_{\pi}^*$ there, the applications
$\mathcal{G}_{\pi}\mathcal{G}_{\pi}^*$, $G_{\pi}G_{\pi}^*$ and $\mathscr{G}_{\pi}\mathscr{G}_{\pi}^*$ coincide on $P(A)\otimes C(\mathbb{G})_{\pi}$ (after applying the suitable
$\Lambda$-maps).

Finally, a simple algebraic computation, coupled with a continuity argument, allows us to conclude that for $x\in \,\!_{\pi}(A\rtimes \mathbb{G})_{\pi}$ and $\xi\in A\otimes
\mathscr{L}^2(\mathbb{G})_{\pi}$, one has \[(\Lambda_{\mathcal{A}}\otimes \iota)(x\xi) = x((\Lambda_{\mathcal{A}}\otimes \iota)\xi),\] and
\[(\Lambda_{\mathscr{A}}\otimes \iota)(x\xi) = x((\Lambda_{\mathscr{A}}\otimes \iota)\xi),\] where the left $\,\!_{\pi}(A\rtimes \mathbb{G})_{\pi}$-module structure on
$\mathscr{A}\otimes \mathscr{L}^2(\mathbb{G})_{\pi}$ is obtained again by making first the identification with $\,\!_{\pi}(A\rtimes \mathbb{G})$. It follows that $G_{\pi}G_{\pi}^* =
\mathcal{G}_{\pi}\mathcal{G}_{\pi}^* = P_{\pi}$.\end{proof}

\section{Freeness of compact quantum group actions}\label{free}

We keep the same notation as in the previous section, thus $\alpha$ is an action of a compact quantum group $\mathbb{G}$ on a C$^*$-algebra $A$.

The action $\alpha$ is said to satisfy \textit{the Ellwood condition}, or simply be \emph{free}, if the following cancellation property holds:
\[\lbrack \alpha(A)(A\otimes 1)\rbrack = A\omin C(\mathbb{G}).\] This condition was introduced in~\cite{Ell1} in a more general setting of actions by locally compact quantum groups.
It is straightforward to check that if $A = C_0(X)$ for some locally compact space $X$, and $\mathbb{G}$ is an ordinary compact group $G$, then the above condition characterizes
precisely the freeness of the action of $G$ on $X$.

Since the subspaces $(\mathcal{A}_\pi)_{\pi \in \Irr \mathbb{G}}$ of $\mathcal{A}$ are orthogonal to each other, and similarly for the supspaces $(\mathscr{L}^2(\mathbb{G})_\pi)_{\pi
\in \Irr \mathbb{G}}$ in $\mathscr{L}^2(\mathbb{G})$, the isometries $G_{\pi}$ and $\mathcal{G}_{\pi}$, add up to respective isometries $G$ and $\mathcal{G}$.

\begin{Prop}\label{PropFree} The following conditions are equivalent.
\begin{thmenum}\item The action $\alpha$ is free.
\item The Hilbert C$^*$-Galois isometry $\mathcal{G}$ is unitary.
\item The C$^*$-Galois isometry $G$ is unitary.
\end{thmenum}
\end{Prop}

\begin{proof} Assume the action is free. Since the natural map $(\Lambda_{\mathcal{A}}\otimes \Lambda_{\varphi})$ from $A\omin C(\mathbb{G})$ to $\mathcal{A}\otimes
\mathscr{L}^2(\mathbb{G})$ is contractive, we see that the image of $\mathcal{G}$ is dense in its range. As $\mathcal{G}$ is isometric, it then follows that it is bijective, hence
unitary (cf.~\cite[Theorem~3.5]{Lan1}).

Let us assume that $\mathcal{G}$ is unitary. Then all $\mathcal{G}_{\pi}$ are unitary operators. Proposition~\ref{PropComp} implies that all $G_{\pi}$ are unitary operators as well.
It follows that $G$ is a unitary.

Finally, let us assume $G$ is unitary. Then we have $\lbrack \alpha(A_{\pi})(A\otimes 1)\rbrack = A\otimes C(\mathbb{G})_{\pi}$ for each representation $\pi$ of $\mathbb{G}$. As
$P(A)$ is dense in $A$, it follows that the action of $\mathbb{G}$ on $A$ is free.
\end{proof}

\begin{Cor}[Theorem~\ref{TheoProj}]\label{CorFin} Let $A$ be a unital C$^*$-algebra endowed with a free action of $\mathbb{G}$, and let $\pi$ be a finite dimensional unitary
representation of $\mathbb{G}$.
\begin{thmenum}
\item The space $A_{\pi}$ is a finitely generated right Hilbert $B$-module (and hence a finitely generated and projective right $B$-module).
\item The space $A\smbox \mathscr{H}_{\pi}$ is finitely a generated right Hilbert $B$-module (hence finitely generated and projective as a right $B$-module).
\end{thmenum}
\end{Cor}
\begin{proof} This follows from Theorem~\ref{ThmAdjProj} and the previous proposition.
\end{proof}

Following the case of group actions, we make the following definition of saturatedness.

\begin{Def} We say that $\alpha$ is \emph{saturated} if $\lbrack (A\rtimes \mathbb{G})_{\mathrm{triv}}\cdot(A\rtimes \mathbb{G})_{\mathrm{triv}}^* \rbrack = A\rtimes
\mathbb{G}$.\end{Def}

Note that if $A$ is unital, this simply says that $1\otimes \widehat{\chi}_{\mathrm{triv}}$ is a full projection in $A\rtimes \mathbb{G}$.  In general, this condition says that $(A
\rtimes \mathbb{G})_{\mathrm{triv}}$ is an imprimitivity bimodule between $B$ and $A \rtimes \mathbb{G}$.

We now prove the equivalence between freeness and saturatedness.

\begin{Theorem}[Theorem~\ref{ThmFreeChar}] A compact quantum group action of $\mathbb{G}$ on a (not necessarily unital) C$^*$-algebra $A$ is saturated if and only if it is
free.\end{Theorem}

\begin{proof} Observe first that the saturatedness condition is equivalent to that $\Pi_{\alpha}(\mathcal{K}(\mathcal{A},\mathcal{A}_{\pi}))$ being equal to $\,\!_{\pi}(A\rtimes
\mathbb{G})$ for each representation $\pi$ of $\mathbb{G}$. But this in turn is equivalent with all maps $\mathscr{G}_{\pi}$ having dense range, i.e.~ being unitaries. The theorem
then follows from Proposition~\ref{PropComp} and Proposition~\ref{PropFree}.
\end{proof}

\section{Finite index inclusions of C$^*$-algebras}\label{indx}

Let $B\subseteq C$ be a unital inclusion of unital C$^*$-algebras. Following~\cite{Wat1}, we call this an \emph{inclusion of finite-index type} if there exists a conditional
expectation $E\colon C\rightarrow B$ and a finite set of pairs $(v_i,w_i)\in C\times C$, called a \emph{quasi-basis for} $E$, such that
\[
\sum_i v_iE(w_ix)=x = \sum_i E(xv_i)w_i \quad (x \in C).
\]
Such a conditional expectation is itself called a conditional expectation of finite-index type.  Then, \textit{the index of $E$} is defined as
\begin{equation}\label{EqIndexCondExp}
\Index(E) = \sum_i v_iw_i
\end{equation}
This does not depend on the choice of quasi-basis, and belongs to the center of $C$.

An equivalent characterisation of a conditional expectation of finite-index type is the following.

\begin{Lem}\label{LemHilb} Let $B\subseteq C$ be as above, and $E\colon C\rightarrow B$ a conditional expectation. Then $E$ is of finite-index type if and only if the right $B$-module
$C$, together with the $B$-valued inner product $\langle x,y\rangle_B = E(x^*y)$, is a finitely generated right Hilbert $B$-module.
\end{Lem}

\begin{proof} If $C$ is a finitely generated right Hilbert $B$-module in the way prescribed above, it follows from~\cite[Corollary 3.1.4]{Wat1} and the remark following it that $E$ is
of finite-index type.

Conversely, assume that $E$ is of finite-index type. We know that $C$ is finitely generated as a right $B$-module, and, by~\cite[Proposition 2.1.5]{Wat1}, that there exists a constant
$c>0$ such that \[ \|E(x^*x)\| \leq \|x^*x\| \leq c \|E(x^*x)\|.\] Hence $C$ with its $B$-valued inner product is complete, and becomes a finitely generated right Hilbert $B$-module.
This concludes the proof.
\end{proof}

Assume now again that $A$ is a unital C$^*$-algebra with an action by a compact quantum group $\mathbb{G}$, with $B$ denoting the C$^*$-algebra of $\mathbb{G}$-fixed elements. Let us
fix a representation $\pi$ of $\mathbb{G}$ with a fixed orthogonal basis $\{e_i\}$ for $\mathscr{H}_{\pi}$, and write $\delta_{\pi}(e_i) = \sum_{j} u_{ij}\otimes e_j$ with $u_{ij}\in
C(\mathbb{G})$.

On the one hand, we can twist the coaction $\alpha$ with the representation $\pi$ to obtain the coaction \[\alpha_{\pi}\colon A\otimes B(\mathscr{H}_{\pi}) \rightarrow (A\otimes
B(\mathscr{H}_{\pi}))\omin C(\mathbb{G}),\] which, using the Sweedler notation $\alpha(a) = a_{(0)}\otimes a_{(1)}$, is given by the formula \[\alpha_{\pi}(a\otimes e_ie_j^*) =
\sum_{k,l} a_{(0)}\otimes e_ke_l^* \otimes u_{ki}^*a_{(1)}u_{lj}.\]

On the other hand, we can also consider the following \emph{left} coaction of $C(\mathbb{G})$ on $B(\mathscr{H}_{\pi})$: \[\Ad_{\pi}\colon B(\mathscr{H}_{\pi})\rightarrow
C(\mathbb{G})\otimes B(\mathscr{H}_{\pi})\colon e_ie_j^* \rightarrow \sum_{k,l} u_{ik}u_{jl}^* \otimes e_ke_l^*.\]
We let $A\smbox B(\mathscr{H}_{\pi})$ denote the space $\{ x \in A \otimes B(\mathscr{H}_\pi) \colon (\alpha \otimes \iota)(x) = (\iota \otimes \Ad_{\pi})(x) \}$.

The following lemma follows from a straightforward computation, which we omit.

\begin{Lem} The C$^*$-subalgebras $(A\otimes B(\mathscr{H}_{\pi}))^{\mathbb{G}}$ and $A\smbox B(\mathscr{H}_{\pi})$ of $A\otimes B(\mathscr{H}_{\pi})$ coincide.
\end{Lem}

\begin{Theorem}[Theorem~\ref{ThmFinIndIncl}, first part]\label{TheoFinIn} Let $A$ be a unital C$^*$-algebra, endowed with a free action of $\mathbb{G}$. Then the inclusion $B =
A^{\mathbb{G}} \subseteq (A\otimes B(\mathscr{H}_{\pi}))^{\mathbb{G}}$ is a finite-index type inclusion.\end{Theorem}

\begin{proof} By the previous lemma, it is equivalent to show that $B\subseteq A\smbox B(\mathscr{H}_{\pi})$ is an inclusion of finite-index type. Let us choose a faithful
$\Ad_{\pi}$-invariant state $\theta_{\pi}$ on $B(\mathscr{H}_{\pi})$. Then we can view $B(\mathscr{H}_{\pi})$ as a Hilbert space by the inner product $\langle x,y\rangle =
\theta_{\pi}(x^*y)$, and it is immediate that this turns $\Ad_{\pi}$ into a representation of $\mathbb{G}$.

Consider the map
\[
E\colon A\smbox B(\mathscr{H}_{\pi}) \rightarrow A\colon x\rightarrow (\iota\otimes \theta_{\pi})x.
\]
First, it is faithful, being a restriction of the faithful map $\iota \otimes \theta_\pi$. Next, its image is the intersection $A \otimes \mathbb{C} \cap A \smbox B(\mathscr{H}_\pi) =
B$. Moreover, the $E$-induced inner product $\langle x,y\rangle_B = E(x^*y)$ on $A\smbox B(\mathscr{H}_{\pi})$ coincides precisely with the one we defined in Section~\ref{SecAc}
(viewing $B(\mathscr{H}_{\pi})$ as a Hilbert space as above). Hence, by Corollary~\ref{CorFin}, $A\smbox B(\mathscr{H}_{\pi})$ is a finitely generated right Hilbert $B$-module. The
theorem now follows from Lemma~\ref{LemHilb}.
\end{proof}

Remark that, by the same proof, the theorem holds more generally for any coaction whose associated C$^*$-Galois map is adjointable (or indeed just its $\pi\otimes
\bar{\pi}$-localization is adjointable).

To end, let us show that, when $\pi$ is irreducible, the index of the above inclusion is in fact a scalar, equal to the square of the quantum dimension of $\mathscr{H}_{\pi}$. This
should not be surprising: it is the quantum analogue of the fact that if $V$ is representation of a compact group $G$, and $X$ a compact space with a free action by $G$, the pullback
of $V$ to $X$ by means of the action gives a vector bundle of constant rank the
dimension of $V$. For more on this in a von Neumann algebraic context, see
e.g.~\cite{Was2},\cite{Ban2} and~\cite{Ued1}.

Note that the formula~\eqref{EqIndexCondExp} still makes sense in case $E$ is just a $B$-bimodular map from $C$ to $B$ and the `inclusion map' map $B\rightarrow C$ is not injective.
We will make use of the following lemma, whose proof is very similar to the one for the statement that finite index is stable under the Jones tower construction (cf.~\cite[Proposition
1.6.6]{Wat1}). It will later on allow us to tune down one half an argument of~\cite{Was2}.\\

\begin{Lem}\label{LemIn} Let $B$ be a unital C$^*$-algebra, and $\mathcal{E}$ a right Hilbert $B$-bimodule which is finitely generated as a left and as a right $B$-module. Assume that
$\mathcal{E}$ has a \emph{left} $B$-valued inner product $\,\!_B\langle\,\cdot\,,\,\cdot\,\rangle$ such that $\mathcal{E}$ becomes also a left Hilbert $B$-bimodule (with respect to
the given $B$-bimodule structure on $\mathcal{E}$).

Then the map \[E\colon \mathcal{K}(\mathcal{E})_B \rightarrow B \colon \xi\eta^* \rightarrow \,\!_B\langle \xi,\eta\rangle\] is well-defined and $B$-bimodular (with respect to the
natural map $B\rightarrow \mathcal{K}(\mathcal{E})_B$). Choosing further a finite set of elements $\xi_j,\eta_j,\widetilde{\xi}_i,\widetilde{\eta}_i \in\mathcal{E}$ such that \[\sum_j
\xi_j\langle \eta_j,\xi\rangle_B = \xi = \sum_i \,\!_B\langle\xi,\widetilde{\xi}_i\rangle \widetilde{\eta}_i, \qquad \textrm{for all }\xi\in \mathcal{E},\] the elements \[v_{ij} =
\xi_j\widetilde{\eta}_i^*,\quad w_{ij} = \widetilde{\xi}_i\eta_j^*\] form a quasi-basis for $E$, and hence the index of $E$ is given by \[\Index(E) = \sum_j \xi_j \Big(\sum_i \big
\langle \widetilde{\eta}_i,\widetilde{\xi}_i \big \rangle_B\Big) \eta_j^*.\]\end{Lem}

To be clear, the adjoint operation is taken with respect to the right Hilbert module structure. Note that, if the sum $\sum_i \big \langle \widetilde{\eta}_i,\widetilde{\xi}_i
\big\rangle_B$ equals $c1$ for some scalar $c$, then we have $\Index(E) = c1$.

\begin{proof} We will only verify that one half of the quasi-basis property w.r.t.~ $E$ is satisfied for the pairs $(v_{i j}, w_{i j})$. Note that we can restrict ourselves to the
verification of the identity $T = \sum_{i, j} v_{i, j} E(w_{i, j} T)$ for rank $1$ endomorphisms $T$ of $\mathcal{E}$, as we can obtain the same formula for arbitrary $T \in
\mathcal{K}(\mathcal{E})$ by the linearity.

Suppose that $T = \xi \eta^*$ for some vectors $\xi, \eta$ in $\mathcal{E}$.  Then we have $w_{i j} T = \tilde{\xi}_i \langle \eta_j, \xi \rangle_B \eta^*$.  Hence we can compute
 \[
 \sum_{i, j} v_{i j} E(w_{i j} T) = \sum_{i, j} \xi_j \tilde{\eta}_i^* \big._B \big \langle \tilde{\xi}_i \langle \eta_j, \xi \rangle_B, \eta \big \rangle = \sum_{i, j} \xi_j \big (
 \big._B \big \langle \tilde{\xi}_i \langle \eta_j, \xi \rangle_B, \eta \big \rangle^* \tilde{\eta}_i \big)^*.
 \]
 Note that the inner product satisfies the symmetry ${}_B \langle x, y \rangle^* = {}_B \langle y, x \rangle$, and the compatibility ${}_B \langle x, y a \rangle = {}_B \langle x a^*,
 y \rangle$ with the bimodule structure.  Using these, we may further transform the right hand side of the above to
 \[
 \sum_{i, j} \xi_j \left ( {}_B \big \langle \eta \langle \eta_j, \xi \rangle_B^*,  \tilde{\xi}_i \big \rangle \tilde{\eta}_i \right)^* = \sum_j \xi_j \left ( \eta \langle \eta_j, \xi
 \rangle_B^* \right )^* = \sum_j \xi_j \langle \eta_j, \xi \rangle_B \eta^* = \xi \eta^*,
 \]
 which proves $\sum_{i, j} v_{i j} E(w_{i j} T) = T$.
\end{proof}

Let us now fix a free action of $\mathbb{G}$ on a unital C$^*$-algebra $A$. We also fix an irreducible finite dimensional unitary representation $\pi$ of $\mathbb{G}$. Let us recall
here also the following \emph{strong left invariance} property for $\varphi$: if $g,h\in P(\mathbb{G})$, then we have\[ \varphi(gh_{(2)})\,h_{(1)} = \varphi(g_{(2)}h)\,S(g_{(1)}).\]

\begin{Lem}\label{LemFait} The natural representation of $A\smbox B(\mathscr{H}_{\pi})$ on $A\smbox \mathscr{H}_{\pi}$ is faithful.
\end{Lem}

\begin{proof} Note first that we can realize $\mathcal{K}(A\smbox \mathscr{H}_{\pi})$ as an ideal inside $A\smbox B(\mathscr{H}_{\pi})$, formed by the span of the elements of the form
$\sum_{i,j} x_{i}y_j^*\otimes e_ie_j^*$ with $\sum x_i\otimes e_i$ and $\sum y_i\otimes e_i$ inside $A\smbox \mathscr{H}_{\pi}$. The faithfulness of our representation is then
equivalent to that this ideal equals the whole of $A\smbox B(\mathscr{H}_{\pi})$.

Now choose an orthonormal basis $e_i \in \mathscr{H}_{\pi}$, and write $\delta_{\pi}(e_i) = \sum_j u_{ij}\otimes e_j$. Take $x\in P(A)$. Then, using that $S(u_{ij}^*)=u_{ji}$ together
with strong left invariance, one finds that $\sum_j \varphi(u_{ji}^*x_{(1)}) x_{(0)} \otimes e_j$ lies in $A\smbox \mathscr{H}_{\pi}$, where we have used again the Sweedler notation
for the coaction $\alpha$. Hence $\mathcal{K}(A\smbox \mathscr{H}_{\pi})$ contains all elements of the form \[\sum_{k,l} \varphi(u_{ki}^*x_{(1)})\varphi(y_{(1)}^*u_{lj})
x_{(0)}y_{(0)}^* \otimes e_ke_l^*.\] But we can write this in the form \[\sum_{k,l} \varphi(u_{ki}^*x_{(2)})\varphi(S^{-1}(x_{(1)})(x_{(0)}y^*)_{(1)}u_{lj})\,(x_{(0)}y^*)_{(0)}
\otimes e_ke_l^*.\] Now from the proof of Propositions~\ref{PropAd} and~\ref{PropFree}, it follows that we can express any $a\otimes h$ with $a\in P(A)$ and $h\in P(\mathbb{G})$ as a
linear combination of elements of the form $\alpha(x)(y^*\otimes 1)$. Hence, taking the particular case $a=1$, we see that $\mathcal{K}(A\smbox \mathscr{H}_{\pi})$ contains all
elements of the form \[\sum_{k,l} \varphi(u_{ki}^*h_{(2)})\varphi(S^{-1}(h_{(1)})u_{lj})\, 1 \otimes e_ke_l^*.\] Rewriting this slightly by means of strong left invariance, this
becomes \[\sum_{k,l,p} \varphi(u_{pi}^*h)\varphi(u_{kp}^*u_{lj}) \,1 \otimes e_ke_l^*.\] Using the matrices $Q_{\pi}$ from Section~\ref{SubSecCQG}, the above simplifies to
\[\dim_q(\mathscr{H}_{\pi})^{-1}\Big(\sum_{p}\left \langle e_j,Q_{\pi}e_p\right \rangle \varphi(u_{pi}^*h)\Big) \,1 \otimes 1.\] It is clear that the sum does not vanish for at least
one value for $i,j$ and $h$, proving that $\mathcal{K}(A\smbox \mathscr{H}_{\pi})$ contains the unit of $A\smbox B(\mathscr{H}_{\pi})$, and is thus equal to the latter algebra.
\end{proof}

Remark that the above proof, with a slightly modified last step, works just as well for representations which are not irreducible. Although we will not really need it, let us make the
link then at this point with the theory of eigenmatrices.

\begin{Cor} Let $\pi$ be a representation of $\mathbb{G}$, and choose an orthonormal basis $(e_i)$ for $\mathscr{H}_{\pi}$. Write $\delta(e_i) = \sum_j u_{ij}\otimes e_j$, and write
\[A(\pi) = \Big\{ x= \sum_{ij} x_{ij}\otimes e_{ji}\in M_n(A) \mid \alpha(x_{ij}) = \sum_k x_{ik}\otimes u_{kj}\Big\} \subseteq A\otimes B(\mathscr{H}_{\pi}),\] which is called the
space of $\pi$-eigenmatrices. Then we have
\begin{equation}\label{EqEigenMatEnough}
A(\pi)A(\pi)^* = (A\otimes B(\mathscr{H}_{\pi}))^{\mathbb{G}}.
\end{equation}
\end{Cor}

\begin{proof} It is easily verified that we have an isomorphism of vector spaces \[(A\smbox \mathscr{H}_{\pi})\otimes \mathscr{H}_{\pi}^* \rightarrow A(\pi)\colon \Big (\sum_i
a_{i}\otimes e_i \Big)\otimes e_j^* \rightarrow \sum_i a_i \otimes e_ie_j^*.\] It follows immediately that \[A(\pi)A(\pi)^* = (A\smbox \mathscr{H}_{\pi})(A\smbox
\mathscr{H}_{\pi})^*,\] from which the corollary follows by the previous lemma.
\end{proof}

Peligrad~\cite[Corollary 3.5]{Pel2} showed that, for the case of compact group actions, the condition~\eqref{EqEigenMatEnough} is equivalent to the saturatedness of the action. These
results are also closely related to the \emph{strong monoidality} of the operation $A\smbox -$.

Assume now again $\pi$ irreducible, and consider on $A\smbox \mathscr{H}_{\pi}$ the left Hilbert $B$-module structure by \[\,\!_B\langle \xi,\eta\rangle = (\iota\otimes
\theta_{\pi})\xi\eta^*,\] where $\theta_{\pi}$ is the unique $\Ad_{\pi}$-invariant state on $B(\mathscr{H}_{\pi})$. It is easy to see that the conditional expectation $E\colon A\smbox
B(\mathscr{H}_{\pi})\cong \mathcal{K}(A\smbox \mathscr{H}_{\pi}) \rightarrow B$ from Theorem~\ref{TheoFinIn} corresponds precisely to the one from Lemma~\ref{LemIn}.

\begin{Theorem}[Theorem~\ref{ThmFinIndIncl}, second part] The index of the conditional expectation $(\iota\otimes \theta_{\pi})$ for $B\subseteq (A\otimes
B(\mathscr{H}_{\pi}))^{\mathbb{G}}$ equals $\dim_q(\mathscr{H}_{\pi})^2$, the square of the quantum dimension of $\mathscr{H}_{\pi}$.\end{Theorem}

\begin{proof} As in Section~\ref{SubSecCQG}, let us identify $\mathscr{H}_{\overline{\pi}}$ with $\mathscr{H}_{\pi}^*$ endowed with the modified inner product
\[
\llangle e_i^*,e_j^*\rrangle = \theta_{\pi}(e_ie_j^*).
\]
Then we calculate that \[\Big (\,\!_B \big \langle \sum_i x_i\otimes e_i,\sum_i y_i\otimes e_i \big \rangle \sum_i z_i\otimes e_i \Big)^* = \sum_i z_i^*\otimes e_i^* \Big  \langle
\sum_i y_i^*\otimes e_i^*,\sum_i x_i^* \otimes e_i^* \Big \rangle_B,\] where the $B$-valued inner product on the right is now interpreted in $B\smbox \mathscr{H}_{\overline{\pi}}$.

Choose now orthonormal bases $(e_i)_i$ and $(f_i)_i$ respectively in $\mathscr{H}_{\pi}$ and $\mathscr{H}_{\overline{\pi}}$, and write $\delta(e_i) = \sum_j u_{ij}\otimes e_j$ and
$\delta_{\overline{\pi}}(f_i) = \sum_j v_{ij}\otimes f_j$. Let us choose $h,g \in P(\mathbb{G})$ such that \[\varphi(u_{ij}^* h) = \delta_{ij}, \quad \varphi(v_{ij}^* g) =
\delta_{ij}.\] Finally, let us choose a finite collection of $x_k,y_k,w_l,z_l \in P(A)$ such that \[\sum_k\alpha(x_k)(y_k^*\otimes 1) = 1\otimes h, \quad
\sum_l\alpha(w_l)(z_l^*\otimes 1) = 1\otimes g.\] Then the proof of Lemma~\ref{LemFait}, coupled with the observation at the beginning of the current proof, shows that the elements
\begin{align*} \xi_{i,k} &= \sum_p \varphi(u_{pi}^*x_{k(1)})\,x_{k(0)}\otimes e_p,& \eta_{i,k} &= \sum_p \varphi(u_{pi}^*y_{k(1)})\,y_{k(0)}\otimes e_p,\\ \widetilde{\xi}_{i,l} &=
\sum_p \varphi(z_{l(1)}^*v_{pi})\, z_{l(0)}^* \otimes f_p^*, & \widetilde{\eta}_{i,l} &= \sum_p \varphi(w_{l(1)}^*v_{pi}) w_{l(0)}^* \otimes f_p^*\end{align*} satisfy the hypotheses
of Lemma~\ref{LemIn}. To be clear, if $v = \sum_i c_i e_i^* \in \mathscr{H}_{\overline{\pi}}  = \mathscr{H}_{\pi}^*$, then $v^*$ denotes the vector $\sum_i \overline{c_i} e_i$ in
$\mathscr{H}_\pi$.

Now the same computation as in Lemma~\ref{LemFait} shows that \[\sum_{i,k} \big \langle \widetilde{\eta}_{i,k},\widetilde{\xi}_{i,k}\big \rangle_B = \Big(\sum_i f_if_i^*\Big)1.\]
Therefore we obtain $\Index(E) = \Big(\sum_i f_if_i^*\Big)1$. From the way $\llangle \cdot, \cdot \rrangle$ was defined and~\eqref{EqWorCharProperty}, we see that a possible choice of
$(f_i)_i$ is
\[
f_i = \dim_q(\mathscr{H}_{\pi})^{1/2}\sum_k \langle e_i,Q_{\pi}^{1/2}e_k\rangle e_k^*.
\]
For this choice we can compute $\sum_i f_if_i^* = \dim_q(\mathscr{H}_{\pi}) \textrm{Tr}(Q_{\pi}) = \dim_{q}(\mathscr{H}_{\pi})^2$, finishing the proof.
\end{proof}

\emph{Acknowledgements}: We would like to thank P. Hajac, A. Skalski and W. Szyma\'{n}ski for valuable discussions. We also thank P. Hajac for setting up the research group `New
Results in Noncommutative Topology' at the Banach Center (IMPAN, Poland), of which the authors were members and where the collaboration on this article started.

\end{document}